\documentclass[12pt,a4paper]{article}
\usepackage{amssymb,amsmath,amsthm}
\newtheorem{theorem}{\indent {\bf Theorem}}
\newtheorem{lemma}{\indent {\bf Lemma}}
\theoremstyle{remark}

\newtheorem{Corollary}{Corollary}
\begin{document}
\begin{center}
{\Large\textbf{Estimate of a Trigonometrical Sum Involving
Naturals with Binary Decompositions of a Special Kind}}
\end{center}
\bigskip
\begin{center}
{\large \textbf{K.M. \'{E}minyan}}
\end{center}
\bigskip

\begin{abstract}
Let $\mathbb{N}_0$ be a class of natural numbers whose binary
decompositions has even number of 1. We estimate of the sum
$\sum\limits_{n\in \mathbf{N}_0,n\le X}\exp(2\pi i \alpha n^2)$.
 \end{abstract}

\bigskip
\begin{center}
\large 1. Introduction
\end{center}

Consider the binary decomposition of a positive integer $n$:
$$
n=\sum_{k=0}^\infty\varepsilon_k2^k,
$$
where $\varepsilon_k=0,1$ and $k=0,1,\ldots$

We split the set of positive integers into two nonintersecting
classes as follows:
$$
\mathbb{N}_0=\{n\in \mathbb{N},\quad
\sum_{k=0}^\infty\varepsilon_k\equiv 0\pmod 2\},\quad
\mathbb{N}_1=\{n\in \mathbb{N},\quad
\sum_{k=0}^\infty\varepsilon_k\equiv 1\pmod 2\}.
$$

In 1968, A.O. Gel'fond \cite{G} obtained the following theorem:
\textit{for the number of integers $n$, $n\le X$, satisfying the
conditions $n\equiv l\pmod m$, $n\in \mathbf{N}_j$ ($j=0,1$), the
following asymptotic formula is valid:}
\begin{equation}\label{AOG}
    T_j(X,l,m)=\frac{X}{2m}+O(X^\lambda),
\end{equation}
where $m$, $l$ are any naturals and $\lambda=\frac{\ln 3}{\ln
4}=0,7924818\ldots$

Suppose that
$$
\varepsilon(n)=\left\{%
\begin{array}{ll}
    1 & \hbox{for $n\in \mathbf{N}_0$,} \\
    -1 & \hbox{otherwise.} \\
\end{array}%
\right.
$$

The proof of formula (\ref{AOG}) is based on the estimate
\begin{equation*}\label{AOG1}
    |S(\alpha)|\ll X^\lambda
\end{equation*}
of the trigonometrical sum
$$
S(\alpha)=\sum_{n\le X}\varepsilon(n)e^{2\pi i \alpha n},
$$
which is valid for any real values of $\alpha$.

Note that $S(\alpha)$ is a linear sum.

In present paper we estimate the quadratic sum
$$
S_0(\alpha)=\sum_{n\le X}\varepsilon(n)e^{2\pi i \alpha n^2}.
$$

Sum of this type, as far as we know, isn't yet been considered in
mathematical literature.

Our main result is the following theorem.

\begin{theorem}\label{T1}
Suppose that $X>2$, $\alpha\in \mathbb{R}$. Then for any $c>0$
there exist positive numbers $A$ and $B$ such that if
$$
\alpha=\frac{a}{q}+\frac{\theta}{q^2},\quad (a,q)=1,\quad
|\theta|<1,\quad \log^AX<q\le X\log^{-B}X,
$$
then the inequality
$$
|S_0(\alpha)|\ll X\log^{-c}X
$$
holds.
\end{theorem}

Proof of the theorem 1 is based on the following lemmas.
\bigskip
\begin{center}
{\large 2. Lemmas}
\end{center}

\begin{lemma}\label{l1}
Suppose that $Q\in \mathbb{N}$, $Q\le X$. Then the following
estimate holds:
$$
S_0(\alpha)\ll \frac{X}{\sqrt{Q}}+\{\frac{X}{Q}\sum_{h=1}^{Q-1}
|S(X-h,h,2h\alpha)|\}^{1/2},
$$
where $S(Y,h,\beta)=\sum\limits_{n\le
Y}\varepsilon(n)\varepsilon(n+h)e(\beta n)$, $e(x)=\exp\{2\pi i
x\}$.
\end{lemma}

The proof of this lemma coincides with the proof of a well known
lemma in van der Corput's method.

\begin{lemma}\label{l2}
Suppose that $\beta\in \mathbb{R}$, $h\in \mathbb{N}$,
$2<h\le\log^{c_1}X$, where $c_1>0$ is a constant. Let $N+1$ be a
number of binary digits of $h$, $X_{N+1}=X 2^{-N-1}$,
$\beta_{N+1}=2^{N+1}\beta$.

Then the inequality
$$
|S(X,h,\beta)|\ll h^{\log_23}|S(X_{N+1},1,\beta_{N+1})|+
h^{\log_23}|S(X_{N+1},2,\beta_{N+1})|+h^{\log_23}
$$
holds.
\end{lemma}

\begin{proof}[Proof]
Let $h=1w_Nw_{N-1}\ldots w_1w_0$ be a binary expansion of $h$,
$w_j=0,1$ ($j=0,1,\ldots,N$).

Define numbers $h_j$ with the following binary expansions:
$$
h_0=h=1w_Nw_{N-1}\ldots w_1w_0,\ h_1=1w_Nw_{N-1}\ldots w_1,\
h_2=1w_Nw_{N-1}\ldots w_2,\ldots ,h_N=1w_N.
$$

Define numbers $s_j$ as follows: $s_j=1-2w_j$ ($j=0,\ldots ,N$).

 Grouping summands over even and over odd $n$ and using
obvious formulae $\varepsilon(2n)=\varepsilon(n)$,
$\varepsilon(2n+1)=-\varepsilon(n)$ we have the following
equalities

\begin{equation}\label{f1}
S(X,2h_1,\beta)=(1+e(\beta))S(X2^{-1},h_1,2\beta)+O(1),
\end{equation}
\begin{equation}\label{f2}
S(X,2h_1+1,\beta)=-S(X2^{-1},h_1,2\beta)-e(\beta)S(X2^{-1},h_1+1,2\beta)+O(1).
\end{equation}

Consider the linear combination
$$
x_jS(X_j,h_j,\beta_j)+y_jS(X_j,h_j+1,\beta_j),
$$
where $0\le j\le N$, $X_j=X2^{-j}$, $\beta_j=2^j\beta$.  By
(\ref{f1})--(\ref{f2}) we have
$$
x_jS(X_j,h_j,\beta_j)+y_jS(X_j,h_j+1,\beta_j)=
x_{j+1}S(X_{j+1},h_{j+1},\beta_{j+1})+
y_{j+1}S(X_{j+1},h_{j+1}+1,\beta_{j+1})+
$$
$$
+O(|x_j|+|y_j|),
$$
where
\begin{equation}\label{f3}
x_{j+1}=(s_j+\frac{s_j+1}{2}e(\beta_j))x_j-\frac{s_j+1}{2}y_j,
\end{equation}
\begin{equation}\label{f4}
y_{j+1}=\frac{s_j-1}{2}e(\beta_j))x_j-(\frac{s_j-1}{2}+s_je(\beta_j))
y_j.
\end{equation}

Define $x_0$ and $y_0$
$$
(x_0,y_0)=\left\{%
\begin{array}{ll}
    (1,0), & \hbox{if $h$ is even,} \\
    (0,1), & \hbox{if $h$ is odd.} \\
\end{array}%
\right.
$$

Then we have
$$
S(X,h,\beta)=x_{N+1}S(X_{N+1},1,\beta_{N+1})+
y_{N+1}S(X_{N+1},2,\beta_{N+1})+O(\sum_{j=0}^N(|x_j|+|y_j|).
$$

From the definition of $x_0$ and $y_0$ and from (\ref{f3}),
(\ref{f4}) it follows by induction that
$$
|x_j|\le 3^j,\quad |y_j|\le 3^j
$$
for any $j\ge 0$. Thus we have
$$
|S(X,h,\beta)|\ll h^{\log_23}|S(X_{N+1},1,\beta_{N+1})|+
h^{\log_23}|S(X_{N+1},2,\beta_{N+1})|+h^{\log_23}.
$$
\end{proof}

\begin{lemma}\label{l3}
Suppose that $\gamma\in \mathbb{R}$, $j_0\ge 1$, $Y>2$. Then the
inequality
$$
|S(Y,1,\gamma)|+|S(Y,2,\gamma)|\ll \log Y\sum_{j=1}^{j_0}
|\sum_{n\le Y2^{-j}}e(2^j\gamma n)|+Y2^{-j_0}+\log Y.
$$
\end{lemma}

Proof follows immediately from the equalities
$$
S(Y,1,\gamma)=-\sum_{n\le Y2^{-1}}e(2\gamma
n)-e(\gamma)S(Y2^{-1},1,2\gamma)+O(1),
$$
$$
S(Y,2,\gamma)=S(Y2^{-1},1,2\gamma)+e(\gamma)S(Y2^{-1},2,2\gamma)
+O(1).
$$

\begin{center}
{\large 3. Proof of Theorem 1 }
\end{center}

First apply lemma \ref{l1} with $Q=\log^{2c}X$. Now it is
sufficient to estimate sums $S(X-h,h,2h\alpha)$ for $1\le h\le
Q-1$.

By lemma \ref{l2} we arrive to the inequality
$$
|S(X-h,h,\beta)|\ll h^{\log_23}|S(X_{N+1},1,\beta_{N+1})|+
h^{\log_23}|S(X_{N+1},2,\beta_{N+1})|+h^{\log_23}.
$$

Then by lemma \ref{l3} we have
$$
|S(X-h,h,2h\alpha)|\ll h^{\log_23}\log^2Xj_0\max_{1\le j\le
j_0}|\sum_{n\le
Y2^{-j}}e(2^j\beta_{N+1}n)|+Xh^{\log_23}2^{-j_0}+h^{\log_23}\log
X,
$$
where $j_0$ satisfy to inequalities
$$
Q^{\log_23}j_02^{-j_0}\le\log^{-2c}X<Q^{\log_23}(j_0-1)2^{-j_0+1}.
$$

Since
$$
|\sum_{n\le Y2^{-j}}e(2^j\beta_{N+1}n)|\ll \frac{1}{\parallel
2^j\beta_{N+1}\parallel},
$$
where $\parallel x\parallel =\min(\{x\},1-\{x\})$, we get the
inequality
$$
|S(X-h,h,2h\alpha)|\ll h^{\log_23}\log^2X\max_{1\le j\le
j_0}\frac{1}{\parallel 2^j\beta_{N+1}\parallel}+X\log^{-2c}X,
$$
which holds for any $h$, $1\le h\le Q-1$.

Let $q$ satisfies the inequality $q>\log^AX$, where $A>0$ is such
that $2^{N+j_0+2}h<\frac{1}{10}\log^AX$. Then we have
$$
\parallel 2^j\beta_{N+1}\parallel\ge\frac{1}{q}-
\frac{2^{N+j_0+2}h}{q^2}\ge\frac{0.9}{q},\quad \frac{1}{\parallel
2^j\beta_{N+1}\parallel}\le\frac{10q}{9}.
$$

Let $q$ also satisfies the inequality $q\le X\log^{-B}X$, where
$B>0$ is such that
$$
\frac{Q^{\log_23}j_0\log^2X}{\log^BX}\le\log^{-2c}X.
$$

Then for any $h$, $1\le h\le Q-1$, the inequality
$$
|S(X-h,h,2h\alpha)|\ll X\log^{-2c}
$$
holds, and so
$$
|S_0(\alpha)|\ll X\log^{-c}.
$$

\begin{Corollary}
Suppose that $X>2$,$\alpha\in \mathbb{R}$. Then for any $c>0$
there exist positive numbers $A$ and $B$ such that if
$$
\alpha=\frac{a}{q}+\frac{\theta}{q^2},\quad (a,q)=1,\quad
|\theta|<1,\quad \log^AX<q\le X\log^{-B}X,
$$
then the inequality
$$
|\sum_{n\le X,\ n\in \mathbb{N}_0}e(\alpha n^2)|\ll X\log^{-c}X
$$
holds.
\begin{proof}
By definition we have
$$
\sum_{n\le X,\ n\in \mathbb{N}_0}e(\alpha n^2)=\sum_{n\le X}\frac{1+\varepsilon(n)}{2}e(\alpha n^2)=\frac{1}{2}S_0(\alpha)+\frac{1}{2}S_1(\alpha),
$$
where $S_1(\alpha)=\sum\limits_{n\le X}e(\alpha n^2)$.

Sum $S_0(\alpha)$ is estimated in theorem 1. Estimate $S_1(\alpha)$ in a standard way:
$$
|S_1(\alpha)|^2\ll X+\sum_{|h|\le X}|\sum_{n\le X}e(2hn\alpha)|\ll X+\sum_{|h|\le 2X}\min(X,\frac{1}{\parallel\alpha h\parallel})\ll
$$
$$
\ll(\frac{X}{q}+1)(X+q\log q)\ll X^2\log^{-2c}X.
$$
\end{proof}

\end{Corollary}

\renewcommand{\refname}{References}

{\large Moscow State University of Applied Biotechnologies}

E-mail address: eminyan@mail.ru


\begin{thebibliography}{99}
\bibitem{G}
{\emph{ A.O. Gel'fond } Acta Arith., \textbf{13}, pp. 259--265
(1968)}
\end{thebibliography}
\end{document}